\documentclass[11pt]{article}
\usepackage[utf8]{inputenc}
\usepackage{style}
\usepackage{thm-restate}
\usepackage[margin=1in]{geometry}
\usepackage{tikz}
\usepackage{verbatim}
\usepackage{cancel}
\usepackage{titling}
\usepackage{multirow}
\usepackage{cancel}
\usepackage{quantikz}

\usepackage[colorinlistoftodos,prependcaption,textsize=tiny]{todonotes}

\newcommand{\KL}{\mathrm{KL}}

\crefname{enumi}{Step}{Steps}

\providecommand{\ent}{\mathbf{Ent}}

\definecolor{cb-salmon-pink}{RGB}{255, 182, 119}
\definecolor{ref-color}{RGB}{200, 0, 200}

\newcommandx{\unsure}[2][1=]{\todo[linecolor=red,backgroundcolor=red!25,bordercolor=red,#1]{#2}}
\newcommandx{\change}[2][1=]{\todo[linecolor=blue,backgroundcolor=blue!25,bordercolor=blue,#1]{#2}}
\newcommandx{\info}[2][1=]{\todo[linecolor=OliveGreen,backgroundcolor=OliveGreen!25,bordercolor=OliveGreen,#1]{#2}}
\newcommandx{\improvement}[2][1=]{\todo[linecolor=Plum,backgroundcolor=Plum!25,bordercolor=Plum,#1]{#2}}
\newcommandx{\thiswillnotshow}[2][1=]{\todo[disable,#1]{#2}}
\usepackage{style}

\title{Sampling and Identity-Testing Without\\Approximate Tensorization of Entropy}
 \author{
 William Gay\thanks{Carnegie Mellon University. \texttt{wgay@andrew.cmu.edu}}
\and
 William He\thanks{Carnegie Mellon University. \texttt{wrhe@cs.cmu.edu}}
 \and
 Nicholas Kocurek\thanks{Carnegie Mellon University. \texttt{nkocurek@andrew.cmu.edu}}
 \and
 Ryan O'Donnell\thanks{Carnegie Mellon University. \texttt{odonnell@cs.cmu.edu}. Supported in part by a grant from Google Quantum AI.}}
\date{\small\today}

\begin{document}
\maketitle
\allowdisplaybreaks
\begin{abstract}
    Certain tasks in high-dimensional statistics become easier when the underlying distribution satisfies a local-to-global property called \textit{approximate tensorization of entropy} (ATE). For example, the Glauber dynamics Markov chain of an ATE distribution mixes fast and can produce approximate samples in a small amount of time, since such a distribution satisfies a \emph{modified log-Sobolev inequality}. Moreover, identity-testing for an ATE distribution requires few samples if the tester is given \textit{coordinate conditional} access to the unknown distribution, as shown by Blanca, Chen, Štefankovič, and Vigoda~\cite{blanca2023complexity}.

    A natural class of distributions that do \textit{not} satisfy ATE consists of \emph{mixtures} of (few) distributions that do satisfy ATE. We study the complexity of identity-testing and sampling for these distributions. Our main results are the following:
    \begin{enumerate}
        \item  We show fast mixing of Glauber dynamics from a \textit{data-based 
        initialization}, with optimal sample complexity, for mixtures of distributions satisfying modified log-Sobolev inequalities. This extends work of Huang, Koehler, Lee, Mohanty, Rajaraman, Vuong, and Wu~\cite{KoehlerLV24,HuangMRW24} for mixtures of distributions satisfying Poincaré inequalities.
        
        \item Answering an open question posed by Blanca et al., we give efficient identity-testers for mixtures of ATE distributions in the coordinate-conditional sampling access model. We also give some simplifications and improvements to the original algorithm of Blanca et al.
    \end{enumerate}
\end{abstract}

\newpage
\section{Introduction}
\subsection{Approximate Tensorization of Entropy}
Let $\mu$ be a distribution on the discrete product set $\Sigma^n$. If $f$ is a (non-negative) real-valued function on $\Sigma^n$, then the following functional captures the amount of local variation that $f$ has, where ``local'' means with respect to varying a single component of $\Sigma^n$:
\begin{definition}
    We write $\calL_\mu$ for the functional on functions $f:\Sigma^n\to \R_{\geq0}$ given by
    \begin{align*}
        \calL_\mu[f]=&\sum_{i\in[n]}\Ex_{\bx\sim \mu}\sbra{\ent_{\by\sim\mu|_{\bx_{\setminus i}}}\sbra{f(\by)}}.
    \end{align*}
\end{definition}
\noindent Here $\ent_\mu[\cdot]$ is the standard entropy functional with respect to $\mu$. See \Cref{sec:chain rule} for a full definition. 

Every product distribution $\mu$ on $\Sigma^n$ satisfies \textit{tensorization of entropy}, meaning $\calL_\mu[f]\geq \ent_\mu[f]$ for all~$f$. 
Much work in the context of Markov chain mixing is focused on establishing an approximate version of this inequality for distributions of interest:
\begin{definition}\label{def:ATE}
    A distribution $\mu$ on $\Sigma^n$ satisfies \emph{approximate tensorization of entropy (ATE)} with constant~$c^*$ if for all $f:\Sigma^n\to\R_{\geq0}$,
    \begin{align*}
        \ent_\mu[f] &\leq c^*\cdot  \calL_\mu[f].
    \end{align*}
    Note that $c^*\geq 1$ unless $\mu$ is a point mass (in which case both sides are always $0$).
\end{definition}
A motivation for \Cref{def:ATE} is the fact that many statistical tasks related to $\mu$ become easier when $\mu$ satisfies $c^*$-ATE with a small $c^*$ (close to $1$).

\paragraph{Sampling.} One such statistical task is that of approximately sampling from $\mu$, given oracle access to $\mu$ up to proportionality. When $\mu$ satisfies $c^*$-ATE, then it is well-known that the Glauber dynamics Markov chain for $\mu$ mixes in time $\widetilde{O}(c^*n)$, and Markov chain Monte Carlo techniques are very effective for drawing approximate samples from $\mu$. Thus, establishing ATE is instrumental for obtaining optimal mixing times for Glauber dynamic chains for natural distributions on high-dimensional spaces, such as Gibbs distributions of certain spin systems at high temperatures. See, for example, \cite{bobkov2006modified,caputo2015approximate,anari2021entropic,chen2021optimal,chen2022localization,blanca2022mixing,hermon2023modified,chen2023algorithms,chen2024fast}. 

\paragraph{Identity-Testing.} Another statistical task where ATE helps significantly is identity-testing, also known as hypothesis-testing or goodness-of-fit testing. In this setting, a testing algorithm is given some distance measure $D$ between distributions, a threshold parameter $\epsilon$, access to some kind of description of a known (``visible'') distribution~$\mu$, and some kind of sample access to an unknown distribution~$\pi$. The tester must then satisfy the following performance guarantees:
\begin{enumerate}
    \item If $\pi=\mu$ then the tester accepts $\pi$ with high probability.
    \item If $\infdiv{}{\pi}{\mu}\geq\epsilon$ then the tester rejects $\pi$ with high probability.
\end{enumerate}
The study of the complexity of this problem has a long history in statistics \cite{pearson1900x,fisher1966design}. This is a fundamental problem in science, where one wants to confirm that the behavior of some system conforms to that of a purported model for that system. 

Motivated by previous works on identity-testing with alternative access models (see \Cref{sec:related work testing} for a more detailed account of the literature), Blanca, Chen, Stefankovic, and Vigoda \cite{blanca2023complexity} studied the model of coordinate conditional access, which is a common relaxation of subcube conditioning access and pairwise conditional access when $\Sigma$ is binary. In this access model, the tester gets access to samples from the unknown distribution~$\pi$, and also gets access to $\pi|_{x_{\setminus i}}$ for any $x\in\Sigma^n$. Here $x_{\setminus i}=\{y\in\Sigma^n:|x-y|_0\leq1\}$. Blanca et al.\ proved that distributions $\mu$ satisfying ATE admit very efficient identity testers.

More precisely, let the ``General Oracle'' be an oracle that, when queried, outputs a sample drawn from the unknown distribution $\pi$, and let the ``Coordinate Oracle'' be the oracle that, when queried with a pair $(x,i)$, outputs a sample from $\pi|_{x_{\setminus i}}$. 
\begin{theorem}[\cite{blanca2023complexity}, Theorem 4.1]\label{thm:BCSV}
    Let $\mu$ be a distribution on $\Sigma^n$ that is $\eta$-balanced (see \Cref{def:balance}), fully supported, and ATE with constant $c^*$. Assuming 
    \[
        c^* \leq {\poly(n)}, \quad \eta \geq \exp(-\poly(n)),
    \]
    there is a testing algorithm for $\mu$ with access to the Coordinate Oracle and General Oracle having 
    \[
        \text{sample complexity} \leq 
        O\pbra{\frac{c^*n}{\epsilon}}\cdot\log^3\pbra{\frac{n}{\epsilon}}\cdot f(\eta), \quad \text{where } f(\eta) = 
        \begin{cases}
            \log(1/\eta) &\text{if $|\Sigma| = 2$}, \\
            \frac1{\sqrt{\eta}} &\text{if $|\Sigma| \geq 3$}.
        \end{cases}
    \]
    Also, in the $|\Sigma| \geq 3$ case one can reduce the dependence on $1/\eta$ back to logarithmic at the expense of being  quadratic in the other parameters; precisely, one can also achieve
    \[
        \text{sample complexity} \leq 
        O\pbra{\frac{c^*n}{\epsilon}}^2 \cdot\log^2\pbra{\frac{n}{\epsilon}}\cdot \sqrt{|\Sigma|} \cdot \log(1/\eta).
    \]
\end{theorem}
\subsection{Mixtures of ATE Distributions and Our Results}
As described, certain statistical tasks become much easier when the related distribution $\mu$ satisfies approximate tensorization of entropy. This leads us to consider cases in which this property fails. A very natural class of distributions that do not satisfy ATE is the class of mixtures of distributions that \textit{do} satisfy ATE. We consider distributions $\mu$ of the form $\sum_{a\in[k]}\rho(a) \mu_a$, where $\rho$ is a distribution on $[k]$ and each $\mu_a$ satisfies $c^*$-ATE. We call the $\mu_a$'s the \textit{mixture components}.

Our first result concerns the sampling task. It is easy to see that certain mixtures of ATE distributions have exponentially large mixing times for Glauber dynamics (for example the equal mixture of the $0.1$- and $0.9$-biased distributions on $\{0,1\}^n$). However, this is a lower bound on the mixing time from a \textit{worst-case initialization}. We instead show that when the chain is initialized with a \textit{data-based initialization}, it still experiences fast mixing, given that there is enough data for the initialization:
\begin{theorem}
    \label{thm:sampling informal}
    Let $\mu = \sum_{a=1}^k \rho(a)\mu_a$ be a mixture on $\Sigma^n$ with each component satisfying $c^*$-ATE (or, more weakly, satisfying $\tfrac{1}{c^*n}$-MLSI; see \Cref{def:Phi-Sobolev} and \Cref{rem:ATE to MLSI}). Let $\bm\pi$ be an empirical distribution induced by
    \begin{equation}
        m = O(k/\varepsilon + \log(1/\delta)/\varepsilon)
    \end{equation}
    independent samples from $\mu$. Then with probability at least $1-\delta$ over these samples, the Glauber dynamics for $\mu$ warm-started at $\bm\pi$ mixes to $\infdiv{\KL}{P^t\boldsymbol{\pi}}{\mu} \leq \varepsilon$
    in continuous-time $t = c^*n \cdot O\pbra{\log \log (1/\min_x\mu(x)) + \log (1/\varepsilon)}$.
\end{theorem}

Our next result concerns the identity-testing task. We show that with the same access model studied by \cite{blanca2023complexity}, we can efficiently identity-test any $\mu$ that is a mixture of few ATE distributions, answering the open question posed in that paper:

\begin{theorem}\label{thm:testing informal}
    Let $\mu=\sum_{a = 1}^k\rho(a)\mu_a$ be a mixture on $\Sigma^n$ with each component satisfying $c^*$-ATE. 
    Assume also that $\mu$ is $\eta$-balanced (see \Cref{def:balance}). 
    Then there is an identity-testing algorithm for $\mu$ that uses
    \begin{align}\label{eq:sample complexity}
         O\pbra{\frac{c^*n}{\epsilon}} \cdot \log^2\pbra{\frac{c^*n}{\epsilon}} \cdot \log(1/\eta)\cdot \log\log(1/\eta)\cdot \sqrt{\abs{\Sigma}}+O\pbra{\frac{\sqrt{k}\cdot \log(1/\rho^*)}{\epsilon}}
    \end{align}
    calls to the General Oracle and Coordinate Oracle.
\end{theorem}

\begin{remark}\label{remark:don't need full coordinate oracle}
    Our algorithm, and the algorithm of \cite{blanca2023complexity}, do not use the full power of the Coordinate Oracle. The algorithms even work in the setting where the set of calls that the tester can make to the Coordinate Oracle is predetermined as a random set of pairs $\bx\sim\pi$ and $\bi\in[n]$. Moreover, given an efficiently accessible description of the distribution $\mu$, the test has similarly efficient computational complexity. As mentioned in \cite{blanca2023complexity}, this efficient description is equivalent to the ability to efficiently implement single steps of Glauber dynamics for the distribution~$\mu$.
\end{remark}

As mentioned before, in the case where $k=1$, \Cref{thm:sampling informal,thm:testing informal} are essentially known from previous work. In analyzing MCMC algorithms for sampling from $\mu$, one wants to show that the KL-divergence between the state of the Markov chain and $\mu$ is contracting. When analyzing identity-testing algorithms for $\mu$ one wants show that if an unknown distribution $\pi$ has large KL-divergence from $\mu$, then this is captured by the divergences between neighbors in $\Sigma^n$. The ATE inequality then relates the contraction of KL-divergence through the Markov chain and the divergences between neighbors to the original KL-divergence.

However, when $\mu$ is merely a mixture of distributions that individually satisfy ATE, such a local-to-global property fails. That is, the variation between $\mu$ and some alternative distribution $\pi$ is not conserved when zooming in from the entire distribution $\mu$ to $\mu|_{x \setminus i}$.
In fact, when the variation between $\mu$ and $\pi$ is well-captured by the local variation, we are essentially done. 

In order to deal with distributions $\pi$ whose variations are not captured  locally, one needs to identify where exactly the variation \textit{is} captured, and then deal with such cases accordingly. The chain rule for entropy (\Cref{lem:chain rule}) allows us to identify where the variation is captured, and we develop new techniques to deal with such cases.

\subsection{Related Work in Sampling}\label{sec:related work sampling}

Our result \Cref{thm:sampling informal} can be seen as an entropic analogue of the works of Koehler, Lee, and Vuong~\cite{KoehlerLV24} and Huang, Mohanty, Rajaraman, and Wu~\cite{HuangMRW24}, both of which show a similar result when each component satisfies Poincaré inequality rather than a modified log-Sobolev inequality.

More concretely, suppose each mixture component in~$\mu$ satisfies a Poincaré inequality. Using higher-order spectral gaps for the Glauber dynamics chain, \cite{KoehlerLV24} were able to establish that reasonably fast mixing in TV distance occurs from a warm-start having a near-minimal number of samples. However, due to their use of Poincaré inequalities, the mixing time they conclude can be suboptimal. They also prove a version of their main theorem with a faster mixing time using standard log-Sobolev inequalities, but this relies on the hypothesis that each mixture component satisfies a log-Sobolev inequality, which is a stronger assumption than an MLSI.

\cite{HuangMRW24} obtains a result similar to \cite{KoehlerLV24}, again under the assumption that each mixture component satisfies a Poincaré inequality. However, their result is quantitatively much weaker in sample complexity: (1)~it has a larger dependence on $k$ and $\varepsilon$; and, (2)~it has a dependence on the minimum mixing weight~$\rho_a$, which does not appear in \cite{KoehlerLV24}. \cite{HuangMRW24} also  suffers from suboptimal mixing time due to use of Poincaré inequalities, but it does achieve mixing in $\chi^2$-divergence, which is stronger than mixing in TV distance.

Our \Cref{thm:sampling informal} simultaneously achieves: (1) optimal mixing times in TV distance in the case that $\mu$ is a product distribution, by using MLSIs rather than Poincaré inequalities; and, (2)~optimal sample complexity in the case where $\mu$ is a mixture of isolated point masses. A key point in our analysis is that we avoid using standard Chernoff bounds, which seem to necessarily introduce a dependence on the minimum mixing weight, and instead employ a bound on the moment generating function (m.g.f.)\ for the KL-divergence of an empirical distribution from the true distribution.

\subsection{Related Work in Testing}\label{sec:related work testing}
Our \Cref{thm:testing informal} fits into a broader line of work on testing distributions with alternative sampling models, which is motivated by the fact that testing algorithms with access to i.i.d.~samples from a high-dimensional distribution often require an exponential number of samples.

In the direction of providing more powerful models of access to distributions in identity-testing, concurrent works of Canonne, Ron, and Servedio~\cite{canonne2014testing} and Chakraborty, Fischer, Goldhirsch, and Matsliah~\cite{chakraborty2013power} introduced the problem of identity-testing with access to conditional samples from $\pi$. That is, instead of getting access to i.i.d.~samples from the unknown distribution~$\pi$, the tester can specify a subset $S$ of the underlying set of outcomes and receive samples from $\pi|_S$. These papers showed that \emph{any} distribution admits an identity-tester in this model whose sample complexity is $\mathrm{poly}(1/\epsilon)$, where $\epsilon$ denotes the minimum distance of distributions that should be rejected with high probability. Subsequent work \cite{falahatgar2015faster} continued the study of testers with general conditional samples. 

Other access models to the unknown distribution $\pi$ have also been considered. For example, \cite{canonne2014aggregated} considered the dual query and cumulative dual query models in which one can explicitly query the probability density function on points and subsets of the universe, respectively. See also later works \cite{caferov2015optimal,narayanan2023estimating}.

A model in which one can sample from~$\pi$ conditioned on an arbitrary subset of the universe is rather strong; in many settings it might be unclear how to simulate/obtain such samples. There are, however, natural weakenings of this access model in the high-dimensional setting. In particular, suppose $\Sigma^n$ is the set of configurations of a system of $n$~particles (or individuals, organisms, etc.).
Then it might be significantly more reasonable to obtain samples from ``subcubes'' of~$\Sigma^n$, meaning conditional distributions in which a subset of the particles have fixed states. Indeed, this subcube conditioning model was introduced by Bhattacharyya and Chakraborty \cite{bhattacharyya2018property}, who showed that $\widetilde{O}(n^2)$ subcube-conditioned samples suffice for identity-testing of arbitrary distributions on $\Sigma^n$. Subsequent works \cite{canonne2021random,chen2021learning} provided improvements in cases where assumptions on the visible distributon $\mu$ are made. 

Another realistic relaxation of arbitrary conditioning studied was that of pairwise conditioning, or conditioning under subsets of size~$2$, which was also studied in \cite{canonne2014testing}. Narayanan~\cite{narayanan2021tolerant} provided testing algorithms for arbitrary distributions with complexity $\widetilde{O}\pbra{\frac{\sqrt{n}}{\epsilon^2}}$.

We provide some further motivation for the coordinate conditional sampling access model not mentioned in \cite{blanca2023complexity}. Our motivation is based on practical matters, and we suggest a situation in which one might be able to easily simulate coordinate conditional access in cases where subcube conditional access and pairwise conditional access might not be easily simulable. In our situation, we again regard $\Sigma^n$ as a configuration of $n$ particles, and we think of $\pi$ as some distribution on the set of configurations of those $n$ particles, potentially the Gibbs distribution for some set of interactions between the particles. 

A natural model for the evolution of the configuration over time is that of Glauber dynamics. In the Glauber dynamics process, given a current configuration $x\in\Sigma^n$, the next configuration is chosen from $\pi|_{x_{\setminus i}}$. Assume the ability to i)~arbitrary fix configurations of particles; and ii)~simulate Glauber dynamics for a distribution~$\pi$. Then to simulate access to $\pi|_{x_{\setminus i}}$, one can repeatedly initialize the system to~$x$ and simulate one step of Glauber dynamics for $\pi$ starting from~$x$. Any step that updates the $j$th particle for $j\neq i$ is ignored, and the result given by an update to the $i$th particle is a sample from $\pi|_{x_{\setminus i}}$. Note here that we assume that we can tell when a site undergoes resampling, even if the resampling does not result in a new state.
\section{The Chain Rule for Entropy}\label{sec:chain rule}

\subsection{Entropies and Divergences}
The following notion of $\Phi$-entropy was introduced in \cite{chafai2004entropies}:
\begin{definition}
    Let $\Phi$ be a smooth and convex function mapping some interval of real numbers to the nonnegative real numbers. Let $\mu$ be a probability distribution on a finite set $\Omega$. The \emph{$\Phi$-entropy} of a function $f:\Omega\to\R$ with respect to $\mu$ is defined to be
    \begin{align*}    {\ent_\mu}^\Phi[f]:=\Ex_{\bx\sim\mu}\sbra{\Phi\pbra{f(\bx)}}-\Phi\pbra{\Ex_{\bx\sim\mu}[f(\bx)]}.
    \end{align*}
\end{definition}

\begin{fact}
    If $\Phi$ is convex then $\ent^\Phi_{\mu}[f]\geq 0$.
\end{fact}

\begin{remark}
    Assume $\Phi(1)=0$ and let $\pi$ and $\mu$ be  probability distributions on~$\Omega$. 
    Then the $\Phi$-entropy functional
    \[  {\ent_\mu}^\Phi\sbra{\frac{\pi}{\mu}}=  \infdiv{\Phi}{\pi}{\mu}
    \]
    is also known as the $\Phi$-divergence\footnote{The notation of ``$f$-divergence" is more common, but to avoid notational overload we use the symbol ``$\Phi$."} between $\pi$ and $\mu$. For example,\begin{align*}
        {\ent_\mu}^{u\mapsto u\log u}\sbra{\frac{\pi}{\mu}} &= \infdiv{\KL}{\pi}{\mu}.
    \end{align*}
    As usual, we use the convention $0\log 0 =0$, and that this function is defiend on $\R_{\geq0}$. The quantity $\infdiv{\KL}{\pi}{\mu}$ may be~$\infty$ (when $\pi \not \ll \mu$).
\end{remark}

We are most often interested in the case $\Phi(u)=u\log u$ throughout the paper, so when we drop the $\Phi$ in the superscript, the case $\ent_{\mu}=\ent^{u\log u}_\mu$ is assumed. An important property of this specific choice of $\Phi$, used in our identity-testing result, is that the resulting entropy functional is $1$-homogeneous:
\begin{fact}\label{fact:1-homogeneity}
    $\ent$ is $1$-homogeneous: That is, if $\alpha$ is a nonnegative scalar, then $\ent_\mu[\alpha f]=\alpha \ent_\mu[f]$.
\end{fact}

\subsection{The Chain Rule}
When $\Phi(u)=u^2$, the following \Cref{lem:chain rule} is known as the law of total variance. When $\Phi(u)=u\log u$, it is known as the chain rule for entropy. Both of these tools are of great use in establishing Poincaré and log-Sobolev inequalities. See, for example, \cite{lee1998logarithmic,salez2021sharp}.

The fact itself is standard (see, e.g., \cite{beigi2018phi}), but we give a proof in \Cref{sec:deferred proofs} for the convenience of the reader:
\begin{lemma}\label{lem:chain rule}
    If $\mu=\sum_{a = 1}^k \rho(a)\mu_a$, then $\ent^\Phi_\mu[f]=\ent^\Phi_{\ba\sim\rho}\sbra{\Ex_{\mu_{\ba}}\sbra{f}}+ \Ex_{\ba\sim\rho}\sbra{\ent^\Phi_{\mu_{\ba}}\sbra{f}}$.
\end{lemma}
\Cref{lem:chain rule} is especially useful when a distribution $\mu$ is a mixture of many distributions. In the case where each of the mixture components satisfies approximate tensorization of entropy, we can use \Cref{lem:chain rule} to show that the local entropy of some function $f$ under the distribution $\mu$ is lower-bounded by the portion of the entropy of $f$ that arises as ``intra-component" entropy.
\begin{lemma}\label{lem:local statistic to global statistic}
    If $\mu=\sum_{a = 1}^k \rho(a)\mu_a$, where each $\mu_a$ satisfies $c^*$-ATE, then for any distribution $\pi$ on~$\Sigma^n$ we have
    \begin{align*}
        c^*\cdot \calL_\mu\sbra{f}&\geq \Ex_{\ba\sim\rho}\sbra{{\underset{\bx\sim \mu_{\ba}}{\ent}\sbra{f(\bx)}}}.
    \end{align*}
\end{lemma}
\Cref{lem:local statistic to global statistic} is proved in \Cref{sec:deferred proofs}. \Cref{lem:local statistic to global statistic} essentially shows that the intra-component contribution to the entropy of $f$ is captured by the local entropy in the case where the mixture components satisfy ATE. In the context of our sampling result, this implies that any initialization that is evenly balanced across the components will experience fast mixing. In the context of our identity-testing result, this implies that any unknown distribution that is evenly balanced across the components yet still far from the target distribution will be rejected by a local tester.

The main contribution of this paper is showing how to also handle the inter-component contribution in our applications. That is, we need to show that data-based initializations are evenly balanced with high probability and that our identity-testing algorithm will reject unbalanced distributions. 

For this purpose it will be helpful to characterize the inter-component entropy $\ent_{\ba\sim\rho}[\Ex_{\mu_{\ba}}[f]]$ in the case that $f = \pi/\mu$ is a density function. We can notice that $g$ defined by $a \mapsto \E_{\mu_a}\sbra{f}$ is itself a density on $[k]$ vis-a-vis $\rho$. More explicitly, define $\rho_\pi$ to be the probability distribution on $[k]$ induced by sampling $\bx \sim \pi$, and then drawing $\ba \sim \rho_{\bx}$ where $\rho_{\bx}$ is the posterior of $\bx$ from~$\mu$ with respect to $\rho$. In other words
    \begin{align*}
        \rho_\pi(a) = \sum_{x \in \Omega} \pi(x) \cdot \frac{\rho(a) \mu_a(x)}{\mu(x)}.
    \end{align*}
    Now we can observe that $g$ is indeed the density $\rho_\pi/\rho$; hence:
    \begin{fact}
        \label{claim:intertropyasdivergence}
        For any mixture $\mu = \sum_{a=1}^k \rho(a)\mu_a$ we have $\ent^\Phi_{\ba\sim\rho}\sbra{\E_{\mu_{\ba}}\sbra{\pi/\mu}} = \infdiv{\Phi}{\rho_\pi}{\rho}$.
    \end{fact}

\section{Sampling from Data-Based Initializations}\label{sec:application sampling}

In this section we prove \Cref{thm:sampling informal}. An important observation made in \cite{HuangMRW24} is that the Glauber dynamics chain for a mixture of distributions satisfying either Poincaré or modified log-Sobolev inequalities satisfies a \textit{weak Poincaré} or \textit{weak modified log-Sobolev inequality} respectively. They then use these weak functional inequalities to infer fast mixing. We first generalize this result to the setting of $\Phi$-Sobolev inequalities.

\begin{definition}
    Let $P$ be a reversible discrete-time Markov operator on $\Omega$ with stationary distribution $\mu$. Then the Dirichlet form with respect to $P$ is defined on functions $f, g : \Omega \to \R$ as
    \begin{align*}
        \calE_P(f, g) = \E_{\bx \sim \mu} \E_{\by \sim_P \bx} \sbra{(f(\boldsymbol{x})-f(\boldsymbol{y}))(g(\boldsymbol{x})-g(\boldsymbol{y}))}.
    \end{align*}
\end{definition}

We will study inequalities relating the Dirichlet form, which is a measure of local variation, to measures of global variation in the case where the associated Markov chain is the Glauber dynamics chain with respect to a distribution $\mu$ on $\Sigma^n$. The Glauber dynamics chain is the chain in which given a current state $x^{(t)}\in\Sigma^n$ the next state is sampled by sampling a uniform random $\bi\in[n]$ and then sampling the next state $\bx^{(t)}\sim\mu|_{x^{(t)}_{\setminus i}}$. That is, the chain resamples single coordinates at a time in a way that ensures that $\mu$ is stationary.

\begin{definition}\label{def:Phi-Sobolev}
    A distribution $\mu$ on $\Sigma^n$ satisfies a \emph{$\Phi$-Sobolev inequality} with constant $c^*$ if for all $f : \Sigma^n \to \R_{\geq 0}$,
    \begin{align}\label{eq:MLSI}
        {\ent_\mu}^\Phi\sbra{f} \leq c^* \cdot \calE_P(f, \Phi'(f)).
    \end{align}
    Here $P$ is the Glauber dynamics chain associated to $\mu$.
\end{definition}

The importance of this notion comes from the fact that the right-hand side of \Cref{eq:MLSI} governs the decay of the $\Phi$-entropy between a Markov chain's current distribution and the stationary distribution through time.

\begin{remark}
    Let $\Phi(x) = x\log x$ and notice that $\Phi' = x \mapsto 1+\log x$. Since $\mathcal{E}_P$ is translation-invariant, we have $\calE_P(f, \Phi'(f)) = \mathcal{E}_P(f, \log f)$, and therefore \Cref{eq:MLSI} is the \emph{modified log-Sobolev inequality}. Similarly, when $\Phi(x) = x^2$ we have $\calE_P(f, \Phi'(f)) = 4\mathcal{E}(f, f)$, and \Cref{eq:MLSI} is (up to factor $4$) the \emph{Poincaré inequality}.
\end{remark}

\begin{remark}\label{rem:ATE to MLSI}
    It is known that $c^*$-ATE implies $\tfrac{2}{c^*n}$-MLSI. See Proposition 1.1 in \cite{caputo2015approximate}.
\end{remark}

The following \Cref{lem:maintrick} is essentially a $\Phi$-entropic generalization of Theorem~4.5 from \cite{HuangMRW24}. See \Cref{sec:deferred proofs} for the proof.
\begin{lemma}
    \label{lem:maintrick}
    Let $\mu = \sum_{a = 1}^k \rho(a)\mu_a$ be a mixture of distributions on $\Sigma^n$ with each $\mu_a$ satisfying a $\Phi$-Sobolev inequality with constant $c^*$.
    Let $P$ be the transition matrix for the Glauber dynamics for $\mu$ and $P_t$ be the associated continuous-time Markov operator. Then, for any initial distribution~$\pi$ we have
    \begin{align}\label{eq:weak MLSI mixing}
        \infdiv{\Phi}{P_t\pi}{\mu} \leq (1-1/c^*n)^t \cdot \infdiv{\Phi}{\pi}{\mu} + \Ex_{\bs}\sbra{ {\ent_{\ba \sim \rho}}^\Phi\sbra{\E_{P_{\bs}\pi}\sbra{\frac{\mu_{\ba}}{\mu}}}},
    \end{align}
    where $\bs$ is some random variable supported on $[0,t]$.
\end{lemma}
\Cref{lem:maintrick} shows that as long as one initializes Glauber dynamics for $\mu$ at a distribution $\pi$ so that the second term on the right-hand side in \Cref{eq:weak MLSI mixing} is small, the chain will experience fast mixing at the rate of the individual mixtures. The question is then how to choose $\pi$ such that this quantity is indeed small. We show that when $\pi$ is a ``data-based initialization'', meaning the empirical distribution formed by some number of i.i.d.~samples from $\mu$, the inter-component entropy is indeed small by concentration.

Recall \Cref{claim:intertropyasdivergence}, which motivates us to study the inter-component entropy as the $\Phi$-divergence between the empirical posterior distribution $\rho_{\bm\pi}=\frac1m\sum_{j=1}^m\rho_{\bx_j}$ and the mixing weights. If the mixture components are separated, the task becomes to learn a distribution on $[k]$ from samples up to $\varepsilon$ error in $\Phi$-divergence. Furthermore, by exactly characterizing the m.g.f.~of the empirical estimator's KL-divergence, we can use a convexity argument to handle the case when components are not separated. While we focus here on the well-studied case of KL-divergence, to establish a version of \Cref{thm:sampling informal} for any $\Phi$-divergence one generally only needs to establish the sample complexity of the learning task for a reasonable estimator (see \cite{canonne2020shortnotelearningdiscrete} for more on such learning tasks).

\subsection{The Case of KL-Divergence}
Our main result, \Cref{thm:sampling informal}, is to apply this paradigm in the case of KL-divergence. When all $\rho_{\bx_j}$ are indicator vectors (corresponding to the case of disjointly supported mixture components) the result then directly follows by taking $m = \Theta(k/\varepsilon + \log(1/\delta)/\varepsilon)$ in the recent result of Agrawal given here:
\begin{theorem}[\cite{agrawal2020finite}, Theorem I.2]\label{thm:agrawalmain}
For $\varepsilon > \frac{k-1}{m}$ we have that
    \begin{align*}
        \Pr_{\ba_1,\dots,\ba_m\sim\rho}\sbra{\infdiv{\KL}{\avg_{j\in[m]} \delta_{\ba_j}}{\rho} > \varepsilon} &\leq e^{-\epsilon m}\cdot \pbra{\frac{e\epsilon m}{k-1}}^{k-1}.
    \end{align*}
\end{theorem}

This result is established immediately by the following m.g.f.~bound when $\lambda = m - \frac{k-1}{\varepsilon}$.

\begin{theorem}[\cite{agrawal2020finite}, Theorem I.3]\label{thm:agrawal}
For $0\leq \lambda <m$ we have that
    \begin{align*}
        \Ex_{\ba_1,\dots,\ba_m\sim\rho}\sbra{\mathrm{exp}\pbra{\lambda \cdot \infdiv{\KL}{\avg_{j\in[m]} \delta_{\ba_j}}{\rho}}} &\leq \pbra{\frac1{1-\lambda/m}}^{k-1}.
    \end{align*}
\end{theorem}

We apply a simple convexity argument to arrive at the same m.g.f.~bound for the general case of overlapping mixture components:
\begin{lemma}\label{lem:convergence of empirical distributions}
    Let $\bm{\pi}=\frac1m\sum_{j=1}^m\delta_{\bx_j}$, where each $\bx_j$ is sampled i.i.d.~from $\mu=\sum_{a=1}^k\rho(a)\mu_a$. Then we have the following moment generating function bound:
    \begin{align*}
        \Ex_{\bx_1,\dots,\bx_m}\sbra{\mathrm{exp}\pbra{\lambda \cdot \infdiv{\KL}{\rho_{\bm\pi}}{\rho}}} &\leq \pbra{\frac1{1-\lambda/m}}^{k-1}.
    \end{align*}
\end{lemma}
\begin{proof}
    For all $\lambda>0$ we have
    \begin{align*}
        \Ex_{\bx_1,\dots,\bx_m}\sbra{\mathrm{exp}\pbra{\lambda \cdot \infdiv{\KL}{\rho_{\bm{\pi}}}{\rho}}}&= \Ex_{\bx_1,\dots,\bx_m}\sbra{\mathrm{exp}\pbra{\lambda \cdot \infdiv{\KL}{\avg_{j\in[m]}\rho_{\bx_j}}{\rho}}}\\
        &=\Ex_{\bx_1,\dots,\bx_m}\sbra{\mathrm{exp}\pbra{\lambda \cdot \infdiv{\KL}{\Ex_{\ba_j\sim\rho_{\bx_j}}\avg_{j\in[m]} \delta_{\ba_j}}{\rho}}}\\
        &\leq \Ex_{\bx_1,\dots,\bx_m}\Ex_{\ba_j\sim\rho_{\bx_j}}\sbra{\mathrm{exp}\pbra{\lambda \cdot \infdiv{\KL}{\avg_{j\in[m]} \delta_{\ba_j}}{\rho}}}\\
        &=\Ex_{\ba_1,\dots,\ba_m\sim\rho}\sbra{\mathrm{exp}\pbra{\lambda \cdot \infdiv{\KL}{\avg_{j\in[m]} \delta_{\ba_j}}{\rho}}}.
    \end{align*}
    The inequality follows from the convexity of the KL-divergence in the left component and the exponential function when $\lambda>0$. Applying \Cref{thm:agrawal} completes the proof.
\end{proof}
We can transfer this m.g.f.~bound (and accordingly the tail bound) to that of a convex combination of these quantities over the $s$:
\begin{lemma}
    \label{lem:tailboundtransfer}
    Let $\bm{\pi}=\frac1m\sum_{j=1}^m\delta_{\bx_j}$, where the $\bx_j$ are sampled i.i.d.~from $\mu=\sum_{a=1}^k\rho(a)\mu_a$. Then we have the following tail bound for any $\bs$ whenever $\epsilon > \frac{k-1}{m}$:
    \begin{align*}
        \Pr_{\bx_1,\dots,\bx_m}\sbra{\Ex_{\bs}\sbra{\infdiv{\KL}{\rho_{P_{\bs}\bm\pi}}{\rho} > \varepsilon
        }} &\leq e^{-\epsilon m}\cdot \pbra{\frac{e\epsilon m}{k-1}}^{k-1}.
    \end{align*}
\end{lemma}
\begin{proof}
    In general, let $\bX=\Ex_{\bs}\sbra{\bX_{\bs}}$ where each $\bX_i$ is distributed identically. Fixing $\lambda \geq0$, we have:
    \begin{align*}
        \E\sbra{e^{\lambda \bX}} &= \E\sbra{e^{\lambda\Ex_{\bs}\sbra{\bX_{\bs}}}}
        \leq \E\sbra{\Ex_{\bs}\sbra{e^{\lambda \bX_{\bs}}}}
        = \Ex_{\bs}\sbra{\E\sbra{e^{\lambda \bX_{\bs}}}}
        = \E\sbra{e^{\lambda \bX_0}}.
    \end{align*}
    The inequality follows via convexity, and $\bX_0$ is a copy of $\bX_s$. 

    Set $\bX_s=\infdiv{\KL}{\rho_{P_s\bm\pi}}{\rho}=\infdiv{\KL}{\avg_{j\in[m]}\rho_{P_s\delta_{\bx_j}}}{\rho}$ for all $s$. Note that each $\bX_s$ is marginally distributed according to the distribution given by $\infdiv{\KL}{\avg_{j\in[m]}\rho_{\delta_{\bx_j}}}{\rho}$ since each $\bx_j\sim\mu$, and $\mu$ is stationary with respect to~$P$. Whenever $0\leq \lambda < m$ we apply the above and get the bound
    \begin{align*}
        \E\sbra{e^{\lambda \Ex_{\bs}\sbra{\infdiv{\KL}{\rho_{P_{\bs}\bm\pi}}{\rho}}}} &\leq \E\sbra{e^{\lambda \infdiv{\KL}{\rho_{\bm\pi}}{\rho}}} \underset{\text{\Cref{lem:convergence of empirical distributions}}}{\leq}\pbra{\frac1{1-\lambda/m}}^{k-1}.
    \end{align*}
    Taking $\lambda = m -\frac{k-1}{\varepsilon}$ as in \Cref{thm:agrawalmain} finishes the proof.
\end{proof}

Combining \Cref{lem:maintrick} and \Cref{lem:tailboundtransfer} immediately implies this formal version of \Cref{thm:sampling informal}:
\begin{theorem}
    \label{thm:weak mlsi and balance}
    Let $\mu = \sum_{a=1}^k \rho(a) \mu_a$ be a mixture distribution with each component satisfying a modified log-Sobolev inequality with constant $c^*$. Let $P_t$ be the continuous-time Markov operator induced by the Glauber dynamics for $\mu$ and $\boldsymbol{\pi} = \frac{1}{m} \sum_{j=1}^m \delta_{\bx_j}$ for $\bx_1, ..., \bx_m \sim \mu$. Then
    \begin{align*}
        \Pr_{\bx_1, ..., \bx_m}\sbra{\infdiv{KL}{P_t \boldsymbol{\pi}}{\mu} > \varepsilon} \leq \delta,
    \end{align*}
    for $m = \Theta(k/\varepsilon + \log(1/\delta)/\varepsilon)$ and $t = c^*n\cdot \Theta\left(\log\log(1/\min_x\mu_x) + \log(1/\epsilon)\right)$.
\end{theorem}
\begin{proof}
    By \Cref{lem:maintrick}, the fact that $\infdiv{\KL}{\pi}{\mu}\leq \log(1/\min_x\mu(x))$, and the setting of $t$ it suffices to show that for the claimed sample complexity $m$ we have
    \begin{align*}
        \Pr_{\bx_1,\dots,\bx_m}\sbra{\Ex_{\bs}\sbra{\infdiv{\KL}{\rho_{P_{\bs}\bm{\pi}}}{\rho}} > 0.5\epsilon} &\leq \delta.
    \end{align*}
    The use of the correct constant in the setting of $m$ and \Cref{lem:convergence of empirical distributions} imply the desired bound.
\end{proof}

\section{Identity-Testing with Coordinate Conditional Sampling}\label{sec:application testing}
Our second application deals with the problem of identity-testing distributions on $\Sigma^n$, but with access to coordinate conditional samples from the target distribution $\mu$, which is the mixture of some number of distributions sampling ATE. We prove \Cref{thm:testing informal}.

\paragraph{Proof Overview.} As in \Cref{sec:application sampling}, if $\mu$ itself satisfies ATE then our result follows, and all that needs to be done is to handle the case where inter-component entropy exists. Note that \Cref{alg:identity test with coordinate oracle} is more or less the same as the algorithm of \cite{blanca2023complexity}, except for \Cref{weight verification reject} (and the use of an improved KL identity-tester). To motivate the design of our algorithm, consider a distribution $\pi$ that is far from $\mu$ in KL-divergence. Then, by the chain rule (\Cref{lem:chain rule}), we have that either $\Ex_{\ba\sim\rho}{\ent_{\mu_{\ba}}\sbra{\frac{\pi}{\mu}}}$ is large, or $\ent_{\ba\sim\rho}{\Ex_{\mu_{\ba}}\sbra{\frac{\pi}{\mu}}}$ is large. Intuitively, either the intra-component entropy is large, or the inter-component entropy is large. 

In the first case, the algorithm of \cite{blanca2023complexity} rejects $\pi$ with high probability. This does not directly follow from the guarantee of \cite{blanca2023complexity}, but another application of the chain rule allows us to deduce this. See \Cref{sec:local reject} for the analysis of this part of the algorithm..

In the second case, the inter-component entropy is large. However, it may be the case that the intra-component entropy is small. For example, $\pi$ could be a mixture of the $\mu_a$, but with incorrect mixture weights. In this case, it is not clear how to use the coordinate oracle to detect this discrepancy in a generic way, especially if there is some intra-component entropy. However, in this case, we note that we can use posterior sampling of $\ba\sim \rho|\bx$, where $\bx$ are samples from $\pi$, to infer what the effective weights on each component are. Here $\rho|x$ is the distribution on $[k]$ where $(\rho|x)(a) = \frac{\rho(a)\mu_a(x)}{\mu(x)}$.

More formally, the inter-component entropy is equal to the KL-divergence between this posterior distribution when $\bx\sim\pi$ and when $\bx\sim\mu$. Given knowledge of the true distribution $\rho$, we can again use our base KL-divergence tester from \Cref{lem:KL tester} to reject this $\pi$. This is the content of \Cref{local reject} in \Cref{alg:identity test with coordinate oracle}, analyzed in \Cref{sec:weight verification reject}.

As a subroutine, we will need an identity-testing algorithm with respect to KL-divergence error~$\eps$, for $\eta$-balanced distributions on domains of size~$d$ (we will use this algorithm for the cases $d=k$ and $d=|\Sigma|$). 
In~\cite{blanca2023complexity}, two such testers were given and analyzed; a main one with sample complexity $O\pbra{\min\cbra{\frac{\sqrt{d}\ln(1/\eta)}{\eps^2}, \frac{1}{\sqrt{\eta} \eps}}},$\footnote{We remark that the first expression in the min essentially always dominates. Even when its $\eps$ dependence is made linear, as in our \Cref{lem:KL tester}, the second expression is smaller only in the regime $\frac{1}{d \log^2 d} \lesssim \eta \leq \frac{1}{d}$, and then only by a $\log d$ factor.\label{foot}} and one for the special case of $k = 2$ with sample complexity $O\pbra{\frac{\ln(1/\eta)}{\eps}}$.
The main effort there was to get $\log(1/\eta)$ dependence, rather than $1/\eta$ dependence; however, the general case suffers from a quadratic dependence on~$1/\eps$.  Here we note a tester with both linear dependence on~$1/\eps$ and logarithmic dependence on~$1/\eta$ can be recovered by combining two known results from the literature; we prove this in \Cref{sec:deferred proofs}:

\begin{theorem}\label{lem:KL tester}
    Let $q$ be a distribution over a universe~$D$ of size~$d$, with each outcome in $D$ having  probability at least~$\eta$ under~$q$. There is an algorithm $\textsc{KL-Test}(p,q,\eps,\delta)$ that given input $\epsilon,\delta>0$ and access to samples from an unknown distribution $p$ on $D$, draws
    \begin{align*}
        O\pbra{\frac{\sqrt{d}\cdot \log\pbra{1/\eta}\cdot \log\pbra{1/\delta}}{\epsilon}}
    \end{align*}
    samples from $p$ and has the following performance guarantee:
    \begin{enumerate}
        \item If $p=q$ then $\textsc{KL-Test}(p,q,\epsilon,\delta)$ accepts with probability at least $1-\delta$.
        \item If $\infdiv{\KL}{p}{q}\geq \epsilon$ then $\textsc{KL-Test}(p,q,\epsilon)$ accepts with probability at most $\delta$.
    \end{enumerate}
\end{theorem}

\begin{remark}
    The \cite{blanca2023complexity} work shows that in \Cref{lem:KL tester}, we can actually take the minimum of the sample complexity with $O(\frac{1}{\sqrt{\eta} \eps})$.
    For simplicity of exposition, we will not carry around this `min' in our subsequent complexity bounds (for the reasons  described in \Cref{foot}).
\end{remark}

With this tester in hand, we may state our algorithm.

\begin{algorithm}[H]
    \addtolength\linewidth{-5.5ex}
    \vspace{0.5em}
    \textbf{Input:} Coordinate Oracle and General Oracle access (See \Cref{remark:don't need full coordinate oracle}) to distribution $\pi$ on $\Sigma^n$, known mixing weights $\rho(1), \dots, \rho(k)$, descriptions of $c^*$-ATE distributions $\mu_1, \dots, \mu_k$ on $\Sigma^n$, and the assumption that distribution $\mu=\sum_{a=1}^k \rho(a) \mu_a$ is $\eta$-balanced.
    
    \textbf{Output:} ``Accept'' or ``Reject''.
    
    \medskip
    
    $\textsc{Product-Set-KL-Test}(\pi, \mu, \epsilon)$:
    \begin{enumerate}
        \item\label{local reject} Independently draw
        $
            T_1 = O\pbra{\frac{c^*n}{\epsilon}}
        $
        pairs $(\bx, \bi)$, where $\bx\sim \pi$ and $\bi\in [n]$ is uniformly random. For each pair $(\bx,\bi)$, reject if $\textsc{KL-Test}\pbra{\pi|_{\bx_{\setminus \bi}},\mu|_{\bx_{\setminus \bi}},\bm{\theta},0.05\cdot \frac1{T_1}}$ rejects, where
        $
            \bm{\theta}\sim \mathrm{Unif}\pbra{\sbra{0.05\epsilon/c^*,\log(1/\eta)}}.
        $
        If the total number of Coordinate Oracle calls needed to perform these calls to $\textsc{KL-Test}$ exceeds 
        \begin{align*}
            T &= 100\cdot T_1\cdot \underbrace{c_{\textsc{KL-Test}}\cdot  \sqrt{\abs{\Sigma}}\cdot \log(1/\eta)\cdot \log(20T_1) \cdot10\log\pbra{\frac{\log(1/\eta)}{\epsilon/c^*}}}_{\text{bounds expected number of Coordinate Oracle uses per call to $\textsc{KL-Test}$}},
        \end{align*}
        then reject. Here $c_{\textsc{KL-Test}}$ is the constant hidden in the $O$-notation of \Cref{lem:KL tester}.
        \item\label{weight verification reject} 
        Consider the distribution $\rho_\pi$ on $[k]$ defined by drawing $\bx \sim \pi$ and outputting~$a \in [k]$ with probability $\frac{\rho(a)\mu_a(\bx)}{\mu(\bx)}$. (Note the algorithm can simulate draws from~$\rho_\pi$ using the General Oracle for~$\pi$.)
        Reject if $\textsc{KL-Test}\pbra{\rho_\pi,\rho,0.5\epsilon,0.1}$ rejects.
        \item Accept.
    \end{enumerate}
    \caption{Identity-testing of $\mu$ with Coordinate Oracle and General Oracle access.}
    \label{alg:identity test with coordinate oracle}
\end{algorithm}

We now prove that \Cref{alg:identity test with coordinate oracle} is a good identity-tester for the distribution $\mu$, deferring proofs of the auxiliary lemmas \Cref{lem:local reject} and \Cref{lem:weight verification reject} to \Cref{sec:proofs for testing}. 
\begin{theorem}[\Cref{thm:testing informal}, restated]\label{thm:main}
    Let $\mu=\sum_{a = 1}^k \rho(a)\mu_a$ be $\eta$-balanced, where $\rho^*=\min_a\rho(a)$ and each $\mu_a$ satisfies $c^*$-ATE. Then \Cref{alg:identity test with coordinate oracle} uses
    \begin{align*}
         \underbrace{O\pbra{\frac{c^*n}{\epsilon} \cdot \log^2\pbra{\frac{c^*n}{\epsilon}} \cdot \sqrt{\abs{\Sigma}}\cdot \log(1/\eta)\cdot \log\log(1/\eta)}}_{\text{from \Cref{local reject}}}+\underbrace{O\pbra{\frac{\sqrt{k}\cdot \log(1/\rho^*)}{\epsilon}}}_{\text{from \Cref{weight verification reject}}}
    \end{align*}
    calls to the General Oracle and Coordinate Oracle and satisfies the following performance guarantees:
    \begin{enumerate}
        \item If $\pi=\mu$ then $\textsc{Coordinate-Oracle-Test}(\mu,\pi,\epsilon)$ rejects with probability at most $0.4$.
        \item If $\infdiv{\KL}{\pi}{\mu}\geq\epsilon$ then $\textsc{Coordinate-Oracle-Test}(\mu,\pi,\epsilon)$ rejects with probability at least~$0.6$.
    \end{enumerate}
\end{theorem}
\begin{proof}
    Recalling that $T_1 = O\pbra{\frac{c^*n}{\epsilon}}$, \Cref{local reject} contributes at most the following amount of calls to the oracles: 
    \begin{align*}
        T & = 1000\cdot c_{\textsc{KL-Test}}\cdot T_1\cdot \sqrt{\abs{\Sigma}}\cdot \log(1/\eta)\cdot \log(20T_1) \cdot\log\pbra{\frac{c^*\log(1/\eta)}{\epsilon}}.\\
        &= O\pbra{\frac{c^*n}{\epsilon}} \cdot\log\pbra{\frac{c^*n}{\epsilon}}\cdot \sqrt{\abs{\Sigma}}\cdot\log(1/\eta)\cdot \log\pbra{\frac{c^*\log(1/\eta)}{\epsilon}}\\
        &= O\pbra{\frac{c^*n}{\epsilon} \cdot \log^2\pbra{\frac{c^*n}{\epsilon}} \cdot \sqrt{\abs{\Sigma}}\cdot \log(1/\eta)\cdot \log\log(1/\eta)}.
    \end{align*}
    \Cref{weight verification reject} in \Cref{alg:identity test with coordinate oracle} uses the following amount of oracle calls:
    \begin{align*}
        {O}\pbra{\frac{\sqrt{k}\cdot \log(1/\rho^*)}{\epsilon}}.
    \end{align*}
    Summing these gives the final number of oracle calls. 

    Now we bound the failure probabilities. Suppose that $\pi=\mu$ so that $\infdiv{\KL}{\pi}{\mu}=0$. Then the probability that \Cref{local reject} in \Cref{alg:identity test with coordinate oracle} rejects is at most $0.1$ by \Cref{lem:local reject}. The probability that \Cref{weight verification reject} in \Cref{alg:identity test with coordinate oracle} rejects is at most $0.1$ by \Cref{lem:weight verification reject}. Therefore, $\textsc{KL-Test}(\mu,\pi,\epsilon)$ accepts with probability at least $0.8$.

    Now suppose that $\infdiv{\KL}{\pi}{\mu}=\ent_{\mu}\sbra{\frac{\pi}{\mu}}\geq \epsilon$. Using the chain rule (\Cref{lem:chain rule}), we have
    \begin{align*}
        \ent_{\mu}\sbra{\frac{\pi}{\mu}}=\Ex_{\ba\sim \rho}\sbra{\ent_{\mu_{\ba}}\sbra{\frac{\pi}{\mu}}}+\ent_{\ba\sim \rho}\sbra{\Ex_{\mu_{\ba}}\sbra{\frac{\pi}{\mu}}},
    \end{align*}
    and we find that one of the two summands on the right-hand side is at least $0.5\epsilon$. In the first case, we have $\Ex_{\ba\sim \rho}\sbra{\ent_{\mu_{\ba}}\sbra{\frac{\pi}{\mu}}}\geq 0.5\epsilon$ and then \Cref{lem:local reject} shows that \Cref{local reject} rejects with probability at least $0.9$. Otherwise if $\ent_{\ba\sim \rho}\sbra{\Ex_{\mu_{\ba}}\sbra{\frac{\pi}{\mu}}}\geq0.5\epsilon$ then \Cref{lem:weight verification reject} shows that \Cref{weight verification reject} rejects with probability at least $0.9$.
\end{proof}

\subsection{Proofs of \Cref{lem:local reject,lem:weight verification reject}}\label{sec:proofs for testing}
In this section we prove \Cref{lem:local reject,lem:weight verification reject} which were the main tools needed in our application to identity-testing.

As in \cite{blanca2023complexity}, our sample complexity has a dependence on the balancedness of our visible distributions:
\begin{definition}\label{def:balance}
    We say that a distribution $\mu$ on $\Sigma^n$ is \emph{$\eta$-balanced} if for all $x\in \Sigma^n$ we have that the distribution $\mu|_{x_{\setminus i}}$ has minimum probability $\eta$.
\end{definition}


We always regard the distribution $\mu$ on $\Sigma^n$ as a mixture of distributions $\mu=\sum_{a = 1}^k \rho(a)\mu_a$, where each $\mu_a$ satisfies approximate tensorization of entropy with constant $c_a$. Moreover, we assume that $\mu$ is $\eta$-balanced. The following \Cref{fact:eta-balance min probability} also shows that $\mu$ is fully supported whenever $\eta >0$:
\begin{fact}\label{fact:eta-balance min probability}
    Let $\mu$ be an $\eta$-balanced distribution on $\Sigma^n$. Then for all $x\in\Sigma^n$, we have $\mu(x)\geq \eta^n$.
\end{fact}
\begin{proof}
    We induct on $n$. In the base case $n=1$ and the conclusion is immediate. 

    For the inductive step let $\mu$ be an $\eta$-balanced distribution on $\Sigma^n$. For each $b\in \Sigma$, let $S_b=\cbra{x\in \Sigma^n:x_n = b}$. It must be the case that for all $b\in\Sigma$, we have $\mu(S_b)\geq \eta$, since otherwise by averaging there would be $y\in\Sigma^{n-1}$ such that $\mu|_{y_{\setminus n}}$ is not $\eta$-balanced. By the inductive hypothesis, since the distribution $\mu|_{S_b}$ is $\eta$-balanced, for any $x\in S_b$ we have $\mu|_{S_b}(x)\geq \eta^{n-1}$. Then $\mu(x)= \mu(S_b)\mu|_{S_b}(x)\geq \eta \cdot \eta^{n-1}=\eta^n$.
\end{proof}

Throughout this section let $0<\epsilon < n\log\pbra{\frac1{\eta}}$, which is without loss of generality since if $\mu$ is $\eta$-balanced then by \Cref{fact:eta-balance min probability} we have $\min_{x\in\Sigma^n}\mu(x)\geq \eta^n$. For any other distribution $\pi$, we have
\begin{align*}
    \infdiv{\KL}{\pi}{\mu} &= \Ex_{\pi}\sbra{\log\pbra{\frac{\pi}{\mu}}} \leq \max_{x}\pbra{\log\pbra{\pi(x)}{\mu(x)}}\leq n\log\pbra{\frac1{\eta}}.
\end{align*}
Let $\pi$ be an arbitrary distribution on $\Sigma^n$, and recall that $\mu=\sum_{a = 1}^k \rho(a)\mu_a$ is a mixture of $k$ distributions $\mu_1,\dots,\mu_k$. Let $\rho^*=\min_a \rho(a)$.

\subsubsection{Rejection by Local Testing}\label{sec:local reject}

\begin{lemma}\label{lem:complexity limit}
    The probability that the number of calls to the Coordinate Oracle and General Oracle needed exceeds the limit stated in \Cref{local reject} is at most $0.01$.
\end{lemma}
\begin{proof}
    With certainty over the choice of $\bx$ and $\bi$, the distribution $\mu|_{\bx_{\setminus \bi}}$ is has minimum probability at least $\eta$ by definition of $\eta$-balancedness. Then a call to $\textsc{KL-Test}\pbra{\pi|_{\bx_{\setminus \bi}},\mu|_{\bx_{\setminus \bi}},\bm{\theta},0.05\cdot \frac1{T_1}}$, by \Cref{lem:KL tester} requires 
    \begin{align*}
        c_{\textsc{KL-Test}}\cdot\frac{\sqrt{\abs{\Sigma}}\cdot \log(1/\eta)\cdot \log(20T_1)}{\bm{\theta}}
    \end{align*}
    calls to the Coordinate Oracle. To compute the total expected number of samples, we first compute
    \begin{align*}
        \Ex_{\bm{\theta}}\sbra{\frac{1}{\bm{\theta}}}&= \frac1{\log(1/\eta)-0.05\epsilon/c^*}\int_{\frac{0.05\epsilon}{c^*n}}^{\log(1/\eta)}\frac1{\theta} d\theta \\
        &=\frac1{\log(1/\eta)-0.05\epsilon/c^*}\ln\pbra{\frac{c^*\log(1/\eta)}{0.05\epsilon}}\\
        &\leq \frac{10}{\log(1/\eta)}\log\pbra{\frac{c^*\log(1/\eta)}{\epsilon}}.
    \end{align*}
    Here we use that $c^*\geq 1$, and that $\epsilon\leq n\log\pbra{\frac1\eta}$ so that
    \begin{align*}
        \frac{0.05c^*\epsilon}{n} \leq \frac{0.05n\log\pbra{\frac1\eta}}{n}= 0.05\log\pbra{\frac1\eta}.
    \end{align*}
    Then, using that $\eta\leq \frac12$, we bound this by
    \begin{align*}
        \Ex_{\bm{\theta}}\sbra{\frac{1}{\bm{\theta}}}&\leq 10\log\pbra{\frac{\log(1/\eta)}{c^*\epsilon}}.
    \end{align*}
    Therefore, the expected number of samples, over all $T_1$ calls to $\textsc{KL-Test}$ is at most 
    \begin{align*}
        T_1\cdot 10\cdot c_{\textsc{KL-Test}}\cdot \sqrt{\abs{\Sigma}}\cdot \log(1/\eta)\cdot \log(20T_1) \cdot\log\pbra{\frac{\log(1/\eta)}{c^*\epsilon}}.
    \end{align*}
    Markov's inequality shows the result. 
\end{proof}

\begin{lemma}\label{lem:local reject}
    Assume that $\mu$ is $\eta$-balanced. Then \Cref{alg:identity test with coordinate oracle} satisfies the following performance guarantees:
    \begin{enumerate}
        \item If $\pi=\mu$ then \Cref{local reject} in $\textsc{KL-Test}(\mu,\pi,\epsilon)$ rejects with probability at most $0.1$.
        \item If $\Ex_{\ba\sim \rho}\sbra{\ent_{x\sim \mu_{\ba}}\sbra{\frac{\pi(\bx)}{\mu(\bx)}}}\geq 0.5\epsilon$ then \Cref{local reject} in $\textsc{KL-Test}(\mu,\pi,\epsilon)$ rejects with probability at least $0.9$.
    \end{enumerate}
\end{lemma}
\begin{proof}
    Define the random variable $\bY$ to be $\ent_{\bz\sim \mu|_{\bx_{\setminus i}}}\sbra{\frac{\pi|_{x_{\setminus i}}(\bz)}{\mu|_{x_{\setminus i}}(\bz)}}$ for $\bx\sim\pi$ and $\bi\in[n]$ uniform.

    Consider the version of \Cref{local reject} without the sample limit. If $\pi=\mu$ then by the guarantee of \Cref{lem:KL tester} shows that the rejection probability in each of the $T_1$ iterations using this unlimited tester in \Cref{local reject} is at most $0.05\cdot \frac1{T_1}$. Since by \Cref{lem:complexity limit} the limit on the number of Coordinate Oracle calls changes the behavior with probability at most $0.01$, the overall rejection probability is by a union bound at most $T_1 \cdot \frac{0.05}{T_1} + 0.01\leq 0.1$.
    
    Now suppose that $\Ex_{\ba\sim \rho}\sbra{\ent_{x\sim \mu_{\ba}}\sbra{\frac{\pi(\bx)}{\mu(\bx)}}}\geq 0.5\epsilon$. Then by \Cref{lem:conditional flip from 1-homogeneity} and \Cref{lem:local statistic to global statistic} we have
    \begin{align*}
        \Ex\sbra{\bY}
        &=\frac1n\sum_{i\in[n]}\Ex_{\bx\sim \pi}\sbra{\ent_{\bm{z}\sim \mu|_{\bx_{\setminus i}}}\sbra{\frac{\pi|_{\bx_{\setminus i}}(\bz)}{\mu|_{\bx_{\setminus i}}(\bz)}}}=\frac1n\calL_\mu\sbra{\frac{\pi}{\mu}}
        \geq \frac{0.5  \epsilon}{c^*n},
    \end{align*}
    where $c^*=\max_ac_a$ is a constant such that all $\mu_a$ satisfy $c^*$-ATE. 

    Using that $\bY\leq \log(1/\eta)$ with certainty,
    \begin{align*}
        \Pr_{\bm{\theta},\bY}\sbra{\bY\geq \bm{\theta}} &= \frac{1}{\log(1/\eta)-\frac{0.05\epsilon}{c^*n}}\int_{\frac{0.05\epsilon}{c^*n}}^{\log(1/\eta)}\Pr_{\bY}\sbra{\bY\geq\theta}d\theta\\
        &\geq \frac1{\log(1/\eta)}\int_{0}^{\log(1/\eta)}\Pr_{\bY}\sbra{\bY\geq\theta}d\theta-\frac1{\log(1/\eta)}\int_0^{\frac{0.05\epsilon/c^*}{n}}\Pr_{\bY}\sbra{\bY\geq\theta}d\theta\\
        &=\Ex\sbra{\bY} - \frac1{\log(1/\eta)}\int_0^{0.05\epsilon/c^*}\Pr_{\bY}\sbra{\bY\geq\theta}d\theta\\
        &\geq \Ex\sbra{\bY} -\frac{0.05 \epsilon}{c^*n}\\
        &\geq \frac{0.9}{n} \cdot \Ex\sbra{\bY}\\
        &\geq \frac{0.45 \epsilon}{c^*n}.
    \end{align*}
    Therefore, for $T_1$ draws of $\bx\sim\pi$, $\bi\sim[n]$ uniform, and $\bm{\theta}$, the probability that none satisfy $\bm{\theta}\leq \ent_{\bz\sim \mu|_{\bx_{\setminus i}}}\sbra{\frac{\pi|_{x_{\setminus i}}(\bz)}{\mu|_{x_{\setminus i}}(\bz)}}$ is at most 
    \begin{align*}
        \pbra{1 - \frac{0.45 \epsilon}{c^*n}}^{T_1}&\leq 0.04
    \end{align*}
    for a correct choice of constant in the definition of $T_1$.
    
    Let $\mathcal{A}$ be the event that this occurs so that $\Pr\sbra{\mathcal{A}}\leq 0.01$. Then let $\mathcal{B}$ be the event that for some call $\textsc{KL-Test}\pbra{\pi|_{\bx_{\setminus \bi}},\mu|_{\bx_{\setminus \bi}},\bm{\theta},0.05\cdot \frac1{T_1}}$ for which $\bm{\theta}\leq \ent_{\bz\sim \mu|_{\bx_{\setminus i}}}\sbra{\frac{\pi|_{x_{\setminus i}}(\bz)}{\mu|_{x_{\setminus i}}(\bz)}}$, the test accepted. The probability that a single call of this form failed is at most $0.05\cdot \frac1{T_1}$, so a union bound gives $\Pr\sbra{\mathcal{B}}\leq 0.05$. 

    By a union bound, the probability of accepting $\pi$ is then
    \begin{align*}
        \Pr\sbra{\mathcal{A}} + \Pr\sbra{\mathcal{B}} &\leq 0.1.
    \end{align*}
    Thus, the performance guarantee is satisfied.
\end{proof}

\subsubsection{Rejection by Posterior Weight Estimation}\label{sec:weight verification reject}

\begin{lemma}\label{lem:weight verification reject}
    \Cref{alg:identity test with coordinate oracle} satisfies the following performance guarantees:
    \begin{enumerate}
        \item If $\pi=\mu$ then \Cref{weight verification reject} in $\textsc{KL-Test}(\mu,\pi,\epsilon)$ rejects with probability at most $0.1$.
        \item If $\ent_{\ba\sim \rho}\sbra{\Ex_{x\sim \mu_{\ba}}\sbra{\frac{\pi(\bx)}{\mu(\bx)}}}\geq 0.5\epsilon$ then \Cref{weight verification reject} in $\textsc{KL-Test}(\mu,\pi,\epsilon)$ rejects with probability at least $0.9$.
    \end{enumerate}
\end{lemma}
\begin{proof}
    By \Cref{claim:intertropyasdivergence} we have that
    \begin{align*}
        \ent_{\ba\sim\rho}\sbra{\Ex_{\bx\sim \mu_{\ba}}\sbra{\frac{\pi(\bx)}{\mu(\bx)}}}
        &= \ent_{\ba\sim\rho}\sbra{\frac{\rho_\pi(\ba)}{\rho(\ba)}}=\infdiv{\KL}{\rho_\pi}{\rho}.
    \end{align*}
    Then the guarantee of \Cref{lem:KL tester} shows that if $\pi=\mu$ the step \Cref{weight verification reject} rejects with probability at most 0.1. Otherwise, if $\KL(\rho_\pi,\rho)\geq 0.5\epsilon$ then \Cref{weight verification reject} in \Cref{alg:identity test with coordinate oracle} rejects with probability at least $0.9$.
\end{proof}

\begin{lemma}\label{lem:conditional flip from 1-homogeneity}
    The following equality holds for all distribution $\pi$ and $\mu$ on $\Sigma^n$:
    \begin{align*}
        \calL_\pi\sbra{\frac{\pi}{\mu}}=&\sum_{i\in[n]}\Ex_{\bx\sim \pi}\sbra{\ent_{\bm{z}\sim \mu|_{\bx_{\setminus i}}}\sbra{\frac{\pi|_{\bx_{\setminus i}}(\bz)}{\mu|_{\bx_{\setminus i}}(\bz)}}}.
    \end{align*}
\end{lemma}
\begin{proof}
    We directly compute
    \begin{align*}
        \sum_{i\in[n]}\Ex_{\bx\sim \pi}\sbra{\ent_{\bm{z}\sim \mu|_{\bx_{\setminus i}}}\sbra{\frac{\pi|_{\bx_{\setminus i}}(\bz)}{\mu|_{\bx_{\setminus i}}(\bz)}}}
        &=\frac1{|\Sigma|}\sum_{i\in[n]}\sum_{x}\pi(x_{\setminus i})\ent_{\bm{z}\sim \mu|_{x_{\setminus i}}}\sbra{\frac{\pi|_{x_{\setminus i}}(\bz)}{\mu|_{x_{\setminus i}}(\bz)}}\\
        &=\frac1{|\Sigma|}\sum_{i\in[n]}\sum_{x}\pi(x_{\setminus i})\cdot \frac{\mu(x_{\setminus i})}{\pi(x_{\setminus i})}\cdot\ent_{\bm{y}\sim \mu|_{x_{\setminus i}}}\sbra{\frac{\pi(\by)}{\mu(\by)}}\\
        &=\frac1{|\Sigma|}\sum_{i\in[n]}\sum_{x}\mu(x_{\setminus i})\cdot\ent_{\bm{y}\sim \mu|_{x_{\setminus i}}}\sbra{\frac{\pi(\by)}{\mu(\by)}}\\
        &=\calL_\pi\sbra{\frac{\pi}{\mu}}.
    \end{align*}
    The second equality follows by $1$-homogeneity of $\ent[\cdot]$ (\Cref{fact:1-homogeneity}).
\end{proof}

\begin{remark}
    \Cref{lem:conditional flip from 1-homogeneity} is the only place we require the $\Phi$-entropy we use to be $\Phi(u)=u\log u$, since this is the only place we need 1-homogeneity. We leave it open for future work whether one can obtain testers for different $\Phi$-entropies (i.e., other divergences besides KL-divergence) by bypassing the need for \Cref{lem:conditional flip from 1-homogeneity}.
\end{remark}

\bibliographystyle{alpha}
\bibliography{ref.bib}
\appendix\section{Deferred Proofs}\label{sec:deferred proofs}
\subsection{Proof of \Cref{lem:chain rule}}
\begin{proof}[Proof of \Cref{lem:chain rule}]
    We directly compute:
    \begin{align*}
        &{\ent_{\ba\sim\rho}}^\Phi\sbra{\Ex_{\bx\sim \mu_{\ba}}\sbra{f}}+ \Ex_{\ba\sim\rho}\sbra{{\ent_{\bx\sim \mu_{\ba}}}^\Phi\sbra{f}}\\
        ={}& \Ex_{\ba\sim\rho}\sbra{\Phi\pbra{\Ex_{\bx\sim \mu_{\ba}}[f]}}-\Phi\pbra{\Ex_{\ba\sim\rho}\Ex_{\bx\sim \mu_{\ba}}[f]}+\Ex_{\ba\sim\rho}\sbra{\Ex_{\bx\sim \mu_{\ba}}\sbra{\Phi(f(\bx))}}-\Ex_{\ba\sim\rho}\sbra{\Phi\pbra{\Ex_{\bx\sim \mu_{\ba}}[f]}}\\
        ={}& \Ex_{\ba\sim\rho}\sbra{\Ex_{\bx\sim \mu_{\ba}}\sbra{\Phi(f(\bx))}}-\Phi\pbra{\Ex_{\ba\sim\rho}\Ex_{\bx\sim \mu_{\ba}}[f]}\\
        ={}& \Ex_{\bx\sim \mu}\sbra{\Phi(f(\bx))}-\Phi\pbra{\Ex_{\bx\sim \mu}[f]}\\
        ={}& {\ent_\mu}^\Phi[f].\qedhere
    \end{align*}
\end{proof} 

\subsection{Proof of \Cref{lem:local statistic to global statistic}}
\begin{proof}[Proof of \Cref{lem:local statistic to global statistic}]
    We compute
    \begin{align*}
        \calL_\mu\sbra{f}=&\sum_{i\in[n]}\Ex_{\bx\sim \mu}\sbra{\ent_{\by\sim\mu|_{\bx_{\setminus i}}}\sbra{f(\by)}}
        =\sum_{i\in[n]}\Ex_{\ba\sim\rho,\bx\sim \mu_{\ba}}\sbra{\ent_{\by\sim\mu|_{\bx_{\setminus i}}}\sbra{f(\by)}}.
    \end{align*}
    Since $\mu|_{\bx_{\setminus i}}=\Ex_{\ba'\sim \rho|\bx_{\setminus i}}[\mu_{\ba'}|_{\bx_{\setminus i}}]$, we have by the chain rule (\Cref{lem:chain rule}) that the above is bounded below by
    \begin{align*}
        \sum_{i\in[n]}\Ex_{\ba\sim\rho,\bx\sim \mu_{\ba}}\sbra{\Ex_{\ba'\sim \rho|\bx_{\setminus i}}\sbra{\ent_{\by\sim\mu_{\ba'}|_{\bx_{\setminus i}}}\sbra{f(\by)}}}=&\sum_{i\in[n]}\Ex_{\ba\sim\rho,\bx\sim \mu_{\ba},\ba'\sim\rho|\bx_{\setminus i}}\sbra{\ent_{\by\sim\mu_{\ba'}|_{\bx_{\setminus i}}}\sbra{f(\by)}}\\
        =&\sum_{i\in[n]}\Ex_{\ba\sim\rho,\bx\sim\mu_{\ba}}\sbra{\ent_{\by\sim\mu_{\ba}|_{\bx_{\setminus i}}}\sbra{f(\by)}}.
    \end{align*}
    By applying ATE for each $\mu_a$, we can lower-bound this by
    \begin{equation*}
        \Ex_{\ba\sim\rho}\sbra{\sum_{i\in[n]}\Ex_{\bx\sim\mu_{\ba}}\sbra{\ent_{\by\sim\mu_{\ba}|_{\bx_{\setminus i}}}\sbra{f(\by)}}}\geq \Ex_{\ba\sim\rho}\sbra{\frac1{c^*}\cdot \ent_{\bx\sim\mu_{\ba}}\sbra{f(\bx)}}. \qedhere
    \end{equation*}
\end{proof}

\subsection{Proof of \Cref{lem:maintrick}}

\begin{proof}[Proof of \Cref{lem:maintrick}]
        Let $\pi$ be any initial distribution and let $f = \pi/\mu$ be the density function of $\pi$ with respect to $\mu$. The impetus behind the switch to continuous-time $P_t$ rather than discrete-time $P$ is that we immediately get the following continuous characterization of $\Phi$-divergence contraction:
        
        \begin{lemma}[\cite{chafai2004entropies}, Proposition 1]\label{lem:chafai}
            Let $f_t = P_t f$. Then,
            \begin{equation*}
                \frac{d}{dt} \infdiv{\Phi}{P_t \pi}{\mu} = -\calE_P(f_t, \Phi'(f_t)).
            \end{equation*}
        \end{lemma}

        This turns out to be the only piece of the proof missing towards a generalization of Theorem~4.5 from \cite{HuangMRW24}. All that is left is to establish a \emph{weak $\Phi$-Sobolev inequality} for mixtures.

        \begin{definition}
        A distribution $\mu$ on $\Sigma^n$ satisfies a \emph{weak $\Phi$-Sobolev inequality} with constant $c^*$ and error $g : \R^{\Sigma^n}_{\geq 0} \to \R_{\geq 0}$ if for all $f : \Sigma^n \to \R_{\geq 0}$,
        \begin{align*}
            {\ent_\mu}^\Phi\sbra{f} \leq c^* \cdot \calE_P(f, \Phi'(f)) + g(f).
        \end{align*}
    \end{definition}
    Let $f : \Sigma^n \to \R_{\geq 0}$ and observe
        \begin{align*}
            {\ent_\mu}^\Phi\sbra{f} &= \Ex_{\ba\sim\rho}\sbra{{\ent_{\bx\sim \mu_{\ba}}}^\Phi\sbra{f}} + {\ent_{\ba\sim\rho}}^\Phi\sbra{\Ex_{\bx\sim \mu_{\ba}}\sbra{f}}\\
            &\leq c^* \cdot \Ex_{\ba\sim\rho}\sbra{\calE_{P_{\ba}}(f, \Phi'(f))} + {\ent_{\ba\sim\rho}}^\Phi\sbra{\Ex_{\bx\sim \mu_{\ba}}\sbra{f}}\\
            &\leq c^* \cdot \calE_P(f, \Phi'(f)) + {\ent_{\ba\sim\rho}}^\Phi\sbra{\Ex_{\bx\sim \mu_{\ba}}\sbra{f}}.
        \end{align*}
        The first line is the chain rule (\Cref{lem:chain rule}). The second line applies the component-wise $\Phi$-Sobolev inequality. The third line invokes the concavity of the Dirichlet form for Glauber dynamics. That is, we use that
        \begin{align*}
            \calE_P(f, \Phi'(f)) &=\Ex_{\bx}\Ex_{\by\sim_P\bx}\sbra{(f(\bx)-f(\by))(\Phi'(f(\bx))-\Phi'(f(\by)))}\\
            &=\frac1n\sum_{x\sim y}\frac{\mu(x)\mu(y)}{\mu(x)+\mu(y)}(f(x)-f(y))(\Phi'(f(x))-\Phi'(f(y)))\\
            &\geq \frac1n\sum_{x\sim y}\Ex_{\ba}\sbra{\frac{\mu_{\ba}(x)\mu_{\ba}(y)}{\mu_{\ba}(x)+\mu_{\ba}(y)}}(f(x)-f(y))(\Phi'(f(x))-\Phi'(f(y))).
        \end{align*}
        Here the inequality follows from concavity of the map $(a,b)\mapsto \frac{ab}{a+b}$ and the fact that $\Phi'$ is increasing in $f(\cdot)$, so the summands are all positive. 
        
        With this weak $\Phi$-Sobolev inequality and \Cref{lem:chafai} the result follows by the exact proof of Theorem~4.5 from \cite{HuangMRW24}, replacing KL-divergence with $\Phi$-divergence and $\calE_P(f_t, \log f_t)$ with $\calE_P(f_t, \Phi'(f_t))$.
\end{proof}

\subsection{Proof of \Cref{lem:KL tester}}
\begin{proof}[Proof of \Cref{lem:KL tester}]
    Irrespective of $q$ having minimum probability~$\eta$, Theorem~1 of \cite{daskalakis2018distribution}\footnote{See also \cite{acharya2015optimal,buadescu2019quantum}.} gives an algorithm we'll call $H^2\textsc{-Test}(p,q,\eps,\delta)$ that --- given $q, \eps, \delta$ and samples from an unknown~$p$ on~$D$ --- has sample complexity $O\pbra{\frac{\sqrt{d}\cdot\log(1/\delta)}{\epsilon}}$ and the following guarantee:
    \begin{enumerate}
        \item If $\chi^2(p||q)\leq 0.5\epsilon$ (e.g., if $p = q$), $H^2\textsc{-Test}(p,q,\eps,\delta)$ accepts with probability at least $1-\delta$.
        \item If $H^2(p||q)\geq \epsilon$, $H^2\textsc{-Test}(p,q,\eps, \delta)$ accepts with probability at most $\delta$.
    \end{enumerate}
    We also have the following inequality relating Hellinger distance and KL-divergence:
    \begin{align*}
        \infdiv{H^2}{p}{q}\geq \frac1{\log(e^2/\eta)}\cdot \infdiv{\KL}{p}{q}.
    \end{align*}
    (With a slightly different constant factor, this inequality appears in, e.g.,~\cite{BM98}. See~\cite[Prop.~2.12]{flammia2024quantum} for the version above.)
    It follows that we can simply run $H^2\textsc{-Test}(\epsilon/\log(e^2/\eta),p,q,\delta)$ to obtain 
    the desired $\textsc{KL-Test}$.
\end{proof}
\end{document}